\documentclass[12pt]{article}

\usepackage{amscd,color}
\usepackage{amssymb,amsmath,amsthm}
\usepackage[all]{xy}

\newcommand{\J}{\mathcal{J}}
\newcommand{\LL}{\mathcal{L}}
\newcommand{\OO}{\mathcal{O}}

\newcommand{\h}{h}
\newcommand{\D}{\mathcal{D}}
\newcommand{\M}{\mathcal{M}}
\newcommand{\Aut}{\mathrm{Aut}}
\newcommand{\Der}{\mathrm{Der}}

\newcommand{\C}{\mathbb{C}}
\newcommand{\Ch}{\mathbb{C}_h}
\newcommand{\Oh}{\mathcal{O}_h}
\newcommand{\Eh}{\mathcal{E}_h}

\newtheorem{theorem}{Theorem}
\newtheorem{prop}[theorem]{Proposition}
\newtheorem{lemma}[theorem]{Lemma}

\newtheorem{defn}[theorem]{Definiton}

\author{Vladimir Baranovsky, Taiji Chen}

\date{January 3, 2017}
\title{Quantization of vector bundles on Lagrangian subvarieties.}

\begin{document}

\maketitle
\abstract{We consider a smooth Lagrangian subvariety $Y$ in a smooth algebraic variety $X$
with an algebraic symplectic from. For a vector bundle $E$ on $Y$ and a choice 
$\Oh$ of deformation quantization of the structure sheaf $\OO_X$ we 
establish when $E$
admits a deformation quantization to a module over $\Oh$.  If the necessary conditions hold, we describe the set of 
equivalence classes of such quantizations.}
\section*{Introduction.}

Consider a smooth algebraic variety $X$ over a field $\C$ of characteristic zero
with an algebraic symplectic
form $\omega \in H^0(X, \Omega^2_X)$, and 
assume that we are given a deformation quantization 
$\Oh$ of the structure sheaf $\mathcal{O}_X$ which agrees with $\omega$.
This means that $\mathcal{O}_h$ is a Zariski sheaf of 
flat associative $\C[[h]]$-algebras on $X$, for which we can find local 
$\C[[h]]$-module isomorphisms $\eta: \Oh \simeq \mathcal{O}_X[[h]]$ such 
that its product $*$ satisfies
$$
a * b  \equiv ab + \frac{1}{2} h P(da, db) \qquad (mod\ h^2) 
$$
where $a, b$ are local sections of $\mathcal{O}_X$ (viewed as local sections of 
$\Oh$ using $\eta$) and $P  \in H^0(X, \Lambda^2 T_X)$ is the Poisson bivector
obtained from $\omega$ via the isomorphism $T_X \to \Omega^1_X$
induced by the same $\omega$.

Given such data and a coherent sheaf $E$ of $\mathcal{O}_X$-modules, 
we could look for a deformation quantization of $E$ as well. Thus, we want 
a Zariski sheaf $\mathcal{E}_h$ of $\Oh$-modules which is flat over 
$\C[[h]]$, complete in $(h)$-adic topology and such that the
$\Oh$-action reduces modulo $h$ to the original action of $\mathcal{O}_X$ on $E$. 
The usual questions are: does $\mathcal{E}_h$ exist at all, 
if yes then how many such sheaves can we find?

In full generality this is a difficult problem. One possible 
simplification is to assume that $E$ is a direct image of a vector bundle
on a closed smooth subvariety $j: Y \hookrightarrow X$. We will denote this 
bundle by $E$ as well (i.e., abusing notation we think of any sheaf on $Y$
as a sheaf on $X$ using the direct image functor $j_*$).

In general, $\Eh$ will not exist at all. 
The first observation is that 
$Y$ must be coisotropic with respect to the symplectic form $\omega$. In other 
words, the bivector 
$P$ projects to a zero section of $\Lambda^2 N$, where $N$ is 
the normal bundle. See Proposition 2.3.1 in \cite{BG}
 for the explanation why $Y$ has to be coisotropic
(this also follows from the proof of Gabber's Integrability of Characteristics
Theorem). In this paper we assume that $Y$ is in fact Lagrangian (i.e. 
isotropic of dimension $\frac{1}{2} \dim_\C X$). Then $P$ induces a
perfect pairing between the tangent bundle $T_Y$ and the normal bundle
$N$ of $Y$, and hence an isomorphism 
$N^* \simeq T_Y$. The case when $E$ has rank $r = 1$ was considered in 
\cite{BGKP} and here we deal with general $r$.

The second observation is that $E$ must carry a  
structure somewhat  similar to a flat algebraic connection.
One could say that the ``quasi-classical limit" of $\Eh$ is given by $E$ 
together with this additional structure, and it is this quasi-classical
limit which is being deformed, not just $E$.

More details are given in Section 2, and the brief account follows here.
A convenient language to use is that of Picard algebroids on $Y$, 
cf. \cite{BB}, i.e. those
Lie algebroids $\LL$ which fit a short exact sequence 
$$
0 \to \mathcal{O}_Y \to \mathcal{L} \to T_Y \to 0
$$
(the trivialization of the sheaf of the left is chosen and forms a 
part of the structure). Such algebroids are classified by their characteristic
class $c(\LL)$ with values in the truncated de Rham cohomology: 
$$
H^2_F(Y) : = H^2(Y, 0 \to \Omega^1_Y \to \Omega^2_Y \to \ldots)
$$ 
One example of such a sheaf is  $\mathcal{L}(\Oh) = \mathcal{T}or_1^{\Oh} (\mathcal{O}_Y, \mathcal{O}_Y)$. 

\medskip
\noindent
Next, $E$ itself gives an 
Atiyah Lie algebroid $\mathcal{L}(E)$ with its exact sequence
$$
0 \to End_{\mathcal{O}_Y}(E) \to \mathcal{L}(E) \to T_Y \to 0. 
$$
A choice of deformation quantization $\Eh$, or even the isomorphism class 
of $\Eh/h^2 \Eh$ as a module over $\Oh/(h^2)$, gives a morphism of 
Zariski sheaves $\gamma: \mathcal{L}(\Oh) \to \mathcal{L}(E)$ which agrees 
with the Lie bracket but not the $\mathcal{O}_Y$-module structure. One can
change the module structure on the source of $\gamma$, 
also changing its characteristic class in $H^2_F(Y)$, to obtain 
a new Picard algebroid $\LL^+(\Oh)$ and a morphism of Zariski 
sheaves $\gamma^+: \LL^+(\Oh) \to \LL(E)$ which now agrees both with 
the bracket and 
the $\mathcal{O}_Y$-module structure. It also embeds $\mathcal{O}_Y \to 
End_{\mathcal{O}_Y}(E)$ as scalar endomorphisms and descends to identity on $T_Y$.
In this situation, following \cite{BB}, we say that $E$ is a module 
over $\mathcal{L}^+(\Oh)$. 

Existence of such $\gamma^+$ is a non-trivial condition on $E$. We
will see in Section 2, and it is only a slight rephrasing of Section 2.3 in 
\cite{BB}, that in this case the projectivization 
$\mathbb{P}(E)$ has a flat algebraic connection and the refined
first Chern class $c_1(E) = c(\LL(\det E))\in H^2_F(Y)$
satisfies the identity 
$$
\frac{1}{r} c_1(E) = c(\mathcal{L}^+(\Oh)). 
$$

Existence of a full deformation quantization for an $\mathcal{L}^+(\Oh)$-module $E$
is formulated in terms of the non-commutative period map
of \cite{BK}: a particular choice of $\mathcal{O}_h$ gives a class 
$$
[\omega] + h \omega_1 + h^2 \omega_2 + \ldots \in H^2_{DR}(X)[[h]]
$$ 
in the algebraic de Rham cohomology of $X$. We will mostly treat the period map as 
a black box, appealing to rank 1 results of \cite{BGKP} that will serve as a bridge
between the definitions in \cite{BK} and our argument.

By the Lagrangian condition, $[\omega]$ restricts to zero on $Y$. The class
of $c(\LL(\Oh))$ in $H^2_F(Y)$ is a canonical lift of the restriction 
 $j^* \omega_1$ of  $\omega_1$ under the closed 
embedding $j: Y \to X$. We will abuse notation and write $j^* \omega_1$ for that
lift as well (note however that in a number of cases of interest, such as $X$ and
$Y$ being complex projective, $H^2_F$ is a subspace of $H^2_{DR}$ so equations in 
the truncated de Rham cohomology may be viewed as equations in the usual de Rham 
cohomology). The class $\omega_1$ affects the choice of $E$ via the identity 
$$
c(\mathcal{L}^+(\Oh)) = \frac{1}{2} c_1(K_Y) 
+ j^* \omega_1,
$$
established in 
Proposition 4.3.5 and Lemma 5.3.5(ii) of \cite{BGKP}.
 The
restrictions of the remaining classes  are also an important ingredient in the main result of our paper:

\begin{theorem} A rank $r$ vector bundle $E$ on a smooth Lagrangian 
subvariety $j: Y \to X$ admits a deformation quantization if and only if the 
following conditions hold:
\begin{enumerate}
\item $j^* \omega_k = 0$ in $H^2_{DR}(Y)$ for $k \geq 2$;
\item the projectivization $\mathbb{P}(E)$ admits a flat algebraic connection;
\item the refined first Chern class in $H^2_F(Y)$ satisfies 
$$
\frac{1}{r} c_1(E) = \frac{1}{2} c_1(K_Y) 
+ j^* \omega_1;
$$
for the canonical lift of $j^* \omega_1$ to $H^2_F(Y)$ representing the 
class of the Picard algebroid $\mathcal{L}(\Oh) = \mathcal{T}or_1^{\Oh} (\mathcal{O}_Y, \mathcal{O}_Y)$. 
\end{enumerate}
If nonempty, the set of equivalence classes of 
all rank $r$ deformation quantizations on $Y$ for 
various $E$
has a free action of the group $\mathcal{G}$ of isomorphism classes
of $\mathcal{O}_Y[[h]]^*$-torsors with a flat algebraic connection.
The set of orbits for this action may be identified with the space of 
all $PGL(r, \C[[h]])$ bundles with a flat algebraic connection.
\end{theorem}

The paper is organized as follows. In Section 1 we discuss the relevant 
details on Lie algebroids and modules over them. In Section 2 we give 
a definition of Harish Chandra pairs, introduce the main Harish Chandra pair
for this paper and etablish its important algebraic properties. In 
Section 3 we explain how the Harish Chandra description arises naturally from 
the formal Darboux lemma for a deformation quantization, and reformulate
our main problem as a lifting problem for transitive Harish Chandra torsors.
In Section 4 we study the lifting problem in three steps and prove the 
main results. Finally, in Section 5
we discuss some related open questions. 

\bigskip
\noindent
To the best of our knowledge, quantization for square roots of the 
canonical bundle has been discovered (without proof) by M. Kashiwara 
in \cite{Ka} in the framework of complex analytic contact geometry. Later
D'Agnolo and Schapira \cite{DS} established a similar result for Lagrangian
submanifolds of a complex analytic symplectic manifold. In the 
$C^\infty$  context, some closely related constructions can be found in the 
work of Nest and Tsygan \cite{NT}. Obstructions to deformation quantization 
have been also studied by Bordemann in \cite{Bo}. The case of complex
tori and quantization of arbitrary sheaves has been studied by 
Ben-Bassat, Block and Pantev in \cite{BBP}. 
The case an arbitrary line bundle and Lagrangian $Y$ 
was considered in \cite{BGKP}. For general rank $r$, 
results on deformations modulo $\h^3$ were obtained 
in \cite{P} where the relation with projectively 
flat algebraic connections has been discussed.
We apologize for any
possible omissions in the references, asking to view them as a 
sign of ignorance rather than arrogance, as the literature on the
subject is somewhat disorganized. 

\bigskip
\noindent
\textbf{Notation.} To unload notation we will write $\Ch$ for the ring $\C[[h]]$ of 
formal power series in $h$ and $\Ch^*$ for its multiplicative group. The symbols
$GL_h(r)$, $PGL_h(r)$ will stand for the groups of $\Ch$-valued points, and 
similarly for their tangent Lie algebras. 

\bigskip
\noindent
\textbf{Acknowledgements.} We are grateful to V. Gizburg, J. Pecharich and 
D. Kaledin for the useful discussions. The first author is supported
by the Simons Collaboration Grant.

\section{Modules over a Lie algebroid.}

In this section we provide more details on the ``quasiclassical limit" for 
a deformation quantization. To that end, 
denote by $I \subset \mathcal{O}_X$ the ideal sheaf of $Y$ and recall 
the isomorphism $I/I^2 \simeq N^*$ of coherent sheaves on $Y$, 
while also $N^* \simeq T_Y$  since $Y$ is Lagrangian. 
Let $I_h \subset \Oh$ be the 
preimage of $I$ in $\Oh$, with respect to the quotient map 
$\Oh \to \mathcal{O}_X$. Assume that a deformation quantization $\Eh$ is 
given and fixed. We will work locally on $Y$ modulo $h^2$, assuming 
that local splittings $\Oh/(h^2) \simeq \mathcal{O}_Y + 
h  \mathcal{O}_Y$, $\Eh/ h^2 \Eh \simeq E + h E$
are given. For $\Oh$ these exist by \cite{Y} and for $\Eh$ the arugment
is similar, cf. the comments in the last section. 

Choosing a local section $a$ of $\mathcal{O}_X$ and a local section 
$e$ of $E$, we write the deformed action as 
$$
a * e = ae + h \gamma_1 (a, e)
$$
If $a \in I$ we see that $a * \Eh \subset \Eh$. Moreover, if 
$a_1, a_2 \in I$ then modulo $h^2$ we can write $a_1 a_2 = 
a_1 * a_2 - h \frac{1}{2} P(da_1, da_2)$ using our assumption on the 
product in $\Oh$. Since $P(da_1, da_2) \in I$ by the Lagrangian 
assumption on $Y$, this implies that $I^2 * \Eh \subset h^2 \Eh$. 
Therefore $(h \mathcal{O}_Y + I/I^2)$ sends $E \simeq \Eh/h \Eh$
to $E \simeq h \Eh/h^2 \Eh$, with $h a + b$ sending $e$ to 
$a e + \gamma_1(b, e)$.  
Writing out associativity equations $a* (b * e) = (a * b)* e$ and 
$b * (a * e) = (b * a ) * e$ and comparing them we get the two 
conditions 
$$
\gamma_1(b, ae) - a \gamma_1(f, e) = P(db, da) e; \qquad 
\gamma_1 (ab, e) - a \gamma_1(f, e) = \frac{1}{2} P(db, da) e.
$$
Observe also that $b \mapsto P(db, \cdot)$ is exactly the 
isomorphism $I/I^2 \simeq T_Y$. Therefore $ha + b$ acts on $E$ 
by a first order differential operator with scalar principal symbol
and we obtain (locally, at this moment) a map $\gamma: h \mathcal{O}_Y  + I/I^2 
\to \LL(E)$ with values in the sections of the Atiyah algebroid of $E$. 
This map agrees with Lie brackets if its source is given the bracket
induced by $(a, b) \mapsto P(db, da)$, $(b_1, b_2) \mapsto 
P(db_1, db_2)$. If we don's start with $\Eh$, just with a deformation 
modulo $h^2$, we need to assume existence of its extension modulo
$h^3$ to ensure agreement with the bracket. 
The map $\gamma$ is not $\mathcal{O}_Y$-linear but
satisfies 
$$
\gamma(f (ha + b)) - f \gamma(ha + b) = \frac{1}{2}P(db, f).
$$
To globalize this consider 
$$
\mathcal{L}(\Oh) = I_h/(I_h * I_h) \simeq I_Y \otimes_{\Oh} \mathcal{O}_Y 
\simeq \mathcal{T}or_1^{\mathcal{O}_h}
(\mathcal{O}_Y, \mathcal{O}_Y)
$$ 
and observe that $h \Oh \subset I_h$ and hence $h I_h \subset (I_h * I_h)$. 
This leads to a short exact sequence
$$
0 \to \mathcal{O}_Y \to \mathcal{L}(\Oh) \to T_Y \to 0
$$
where we use $\mathcal{O}_Y \simeq h \Oh / h I_h$ and $I_h/(h \Oh + I_h* I_h)
\simeq I/I^2 \simeq T_Y$. In other words, $\LL(\Oh)$ is a
a Picard algebroid on $Y$ in the sense of Section 2 in \cite{BB} with the 
bracket that descends from $(a, b) \mapsto \frac{1}{h}(a*b - b*a)$. Our
local computation above gives a morphism of Zariski sheaves 
$$
\gamma: \mathcal{L}(\Oh) \to \mathcal{L}(E)
$$
which agrees with the Lie bracket but satisfies 
$$
\gamma(fx) - f \gamma(x) = \frac{1}{2} \overline{x}(f)
$$
where $f$ is a locally defined function and $\overline{x}$ is the 
image of the local section $x$ of $\LL(\Oh)$ in $T_Y$. So 
$\gamma$ fails to be a morphism of $\mathcal{O}_Y$-modules. 

To repair the situation we use the fact that Picard algebroids form 
a vector space, i.e. for two algebroids $\LL_1, \LL_2$ and any 
pair of scalars $\lambda_1$, $\lambda_2$ there is a Picard algebroid
$\LL = \lambda_1 \LL_1 + \lambda_2 \LL_2$ and a morphism of sheaves
$$
s_{\lambda_1, \lambda_2}: \LL_1 \times_{T_Y} \LL_2 \to \LL,
$$
cf. Section 2.1. in \cite{BB}, 
which on the subbundle copies of $\mathcal{O}_Y$ is given by $(a_1, a_2)
\mapsto \lambda_1 a_1 + \lambda_2 a_2$, and on the quotient copies 
of $T_Y$ it is given by the identity. The other fact that we use is that
the Atiyah algebroid of the canonical bundle $\LL(K_Y)$ has
a non-$\mathcal{O}_Y$-linear splitting sending a vector 
fields $\partial$ to the Lie derivative $l(\partial)$ on top degree
differential forms, which satisfies
$$
l (f \partial) - f l(\partial) =  \partial(f),
$$
cf. Section 2.4 in \cite{BB}. So we consider the algebroid 
$$
\LL^+(\Oh) = \LL(\Oh) + \frac{1}{2} \LL(K_Y)
$$
where $\LL(K_Y)$ is the Aityah algebroid of the canonical bundle $K_Y$. 
The expression $x \mapsto s_{1, \frac{1}{2}}(x, l(\overline{x}) )$
defines an  isomorphism of Zariski sheaves
$$
\LL(\Oh) \to \LL^+ (\Oh)
$$
which agrees with the bracket, descends to identity on $T_Y$ but
fails to be $\mathcal{O}_Y$-linear in exactly the same way as it 
happens for $\LL(\Oh) \to \LL(E)$. Composing the inverse isomorphism 
with $\gamma$ we obtain a morphism of sheaves of $\mathcal{O}_Y$-modules
$$
\gamma^+: \LL^+(\Oh) \to \LL(E)
$$
which is now a morphism of Lie algebroids on $Y$.

Existence of $\gamma^+$ imposes strong restrictions on $E$. 
For instance, composing with the trace morphism 
$End_{\mathcal{O}_Y}(E) \to \mathcal{O}_Y$
we get a morphism $\mathcal{L}^+(\Oh) \to \mathcal{L}(\det(E))$  
to the Atiyah algebroid of the 
determinant bundle $\det(E)$. By construction, this map is multiplication 
by $r = rk(E)$ on the subbundles $\mathcal{O}_Y$ and identity 
on the quotient bundles $T_Y$. In particular, the map is an isomorphism
of Lie algebroids.
Since by \cite{BB} both Picard algebroids 
have characteristic classes in the truncated 
de Rham cohomology, we get
$$
c_1(E) = c_1(\det(E)) = r \cdot c(\mathcal{L}^+(\Oh))
$$
in $H^2_F(Y)$. Furthermore, taking the quotient of $\mathcal{L}(E)$
by the subbundle $\mathcal{O}_Y \subset End_{\mathcal{O}_Y}(E)$
we obtain the Atiyah algebroid  $\mathcal{L}(\mathbb{P}(E))$ 
of the associated $PGL(r)$ bundle $\mathbb{P}(E)$. By construction, 
$\gamma$ descends to a Lie morphism $T_Y \to \mathcal{L}(\mathbb{P}(E))$ which 
lifts the identity on $T_Y$. In other words, we have a flat algebraic 
connection on $\mathbb{P}(E)$. 

The above discussion can also be reversed (we are adjusting the 
argument in Section 2.3 of \cite{BB}): assume that the equation on 
$c_1$ in $H^2_F(Y)$ is satisfied and we are given a flat connection on 
the $PGL(r)$-bundle $\mathbb{P}(E)$ associated to $E$. 
Then we have a commutative diagram
$$
\xymatrix{ \mathcal{O}_Y \ar[r] \ar[d] & \mathcal{L}^+(\Oh) \ar[d] \ar[r] & T_Y \ar[d] \\
\mathcal{O}_Y \oplus End^0(E)  \ar[r] & \ar[r] \mathcal{L}(\det(E))
\oplus \mathcal{L}(\mathbb{P}(E)) & T_Y \oplus T_Y \\
}
$$
Here the lower row is understood as the sum of Atiyah algebras of 
$\det(E)$ and $\mathbb{P}(E)$. The map into the Atiyah algebra of $\det(E)$ is 
as before (identity on $T_Y$ and multiplication by $r$ on $\mathcal{O}_Y$). 
The map into the Atiyah algebra of $\mathbb{P}(E)$ is the composition of 
the projection to $T_Y$ and the splitting $T_Y \to \mathcal{L}(\mathbb{P}(E))$
given by the flat connection.
Observe that the quotient map $T_Y \to T_Y \oplus T_Y$ is just the 
diagonal morphism $\partial \mapsto (\partial, \partial)$. 

Now a local section of the Atiyah algebra of $E$ can be interpreted as
invariant vector field on the total space of the principal $GL(r)$-bundle
of $E$ and its direct image gives vector fields on the total spaces
of the $\C^*$-bundle and the $PGL(r)$-bundle, corresponding to $\det(E)$ 
and $\mathbb{P}(E)$, respectively. Indeed, these total spaces are 
obtained as quotients by the $SL(r)$ action and the $\C^*$ action, 
respectively. This gives an embedding of sheaves 
$$
\mathcal{L}(E) \to \mathcal{L} (\det(E)) \oplus \mathcal{L}(\mathbb{P}(E)) 
$$
which agrees with brackets and has image equal to the preimage of the 
diagonal in $T_Y \oplus T_Y$. This means that the earlier morphism from 
$\mathcal{L}^+(\Oh)$ lands into $\mathcal{L}(E)$, as required. 
We can summarize the discussion of this section as the following
\begin{prop}
A flat deformation quantization of $E$ modulo $h^2$ which admits an 
extension to a deformation modulo $h^3$, 
gives $E$ a structure of a module over the Lie algebroid $\mathcal{L}^+(\Oh)$.
with the class $\frac{1}{2} c_1(K_Y) + j^* \omega_1$ in $H^2_F(Y)$.
For an arbitrary vector bundle $E$ on $Y$ such a structure is equivalent to 
a choice of an isomorphism $\mathcal{L}^+(\Oh) \simeq \mathcal{L}(\det(E))$
as Lie algebroids on $Y$ which satisfies $f \mapsto r f$ on functions, 
and a flat $PGL(r)$ connection on the $PGL(r)$-bundle
$\mathbb{P}(E)$ associated to $E$. 
\end{prop}

\section{Harish Chandra pairs.}

We recall here some definitions and facts about transitive Harish Chandra torsors
which may be found e.g. in \cite{BK}. 

\begin{defn} A \textit{Harish Chandra} pair $(G, \mathfrak{h})$ over $\C$ is a pair 
consisting of a connected (pro)algebraic group 
$G$ over $\C$, a Lie algebra $\mathfrak{h}$ 
over $\C$ with a $G$-action and an embedding $\mathfrak{g} = Lie(G) \subset \mathfrak{h}$
such that the adjoint action of $\mathfrak{g}$ on $\mathfrak{h}$ is the 
differential of the given $G$-action.
\end{defn}

\bigskip
\noindent
The reasons why Harish Chandra pairs are relevant to our problem will be 
explained in the next section while here we introduce the main Harish Chandra 
pair of this paper. 
Let $n = \frac{1}{2}\dim_\C X$ and consider the Weyl algebra $\D$ isomorphic to
$\mathbb{C}[[x_1,\ldots,x_n,y_1,\ldots,y_n,\h]]$ as a vector space but with relations
$[y_i,x_i]=\delta_{ij}\h,[\h,x_i]=[\h,y_j] = [x_i, x_j] = [y_i, y_j] =0$. We view it 
as a ``formal local model" for  $\mathcal{O}_h$.
The ``formal local model" for $\Eh$ is the $\D$-module $\M_r$  
isomorphic to $\mathbb{C}[[x_1, \ldots, x_n, h]]^{\oplus r}$ as a vector space, on which
$x_i, h$ act by multiplication while $y_j$ acts by $h \frac{d}{dx_j}$.

Now consider the group $\Aut(\D, \M_r)$ 
of automorphisms of the pair $(\D, \M_r)$, 
each consisting of a $\Ch$-algebra automorphism $\Phi: \D \to \D$ preserving the maximal ideal 
$\mathfrak{m} = \langle x_i, y_j, h \rangle$, plus a $\Ch$-linear invertible map 
$\Psi: \M_r \to \M_r$ which satisfies $\Psi(am) = \Phi(a) \Psi(m)$ for $a \in \D$, $m \in \M_r$. 

Its Lie algebra 
of derivations of the pair $(\D, \M_r)$ can be expanded by adding derivations which may not
preserve $\mathfrak{m}$. For instance, although the element $\frac{1}{h} y_j$ is not in $\D$, 
the commutator $[\frac{1}{h} y_j, \cdot]$ does define a derivation of $\D$ which sends
e.g. $x_j$ to 1. The same element acts as $\frac{d}{dx_j}$ on $\M_r$.
Therefore, we consider the Lie algebra $\Der(\D, \M_r)$ which consists of 
pairs $\phi: \D \to \D$, $\psi: \M_r \to \M_r$ of $\Ch$-linear maps, where $\phi$ is a
derivation and $\psi$ satisfies $\psi(am) = \phi(a) m + a \psi(m)$. 
This gives a Harish Chandra pair $\langle \Aut(\D,\M_r),\Der(\D,\M_r)\rangle$.

Note that not all automorphisms $\Phi$ and not all derivations $\phi$ can be extended to  
 $(\Phi, \Psi)$ or $(\phi, \psi)$, respectively. Indeed, the quotient module 
 $M_r = \M_r/ h \M_r$ is annihilated by the ideal $\J \subset \D$ generated by $y_j$ 
 and $h$ (which is the ``formal local model" for $I_h$). Since both $\Psi$
and $\psi$ would have to preserve $h \M_r$ and hence to descend to $M_r$ we conclude that 
$\Phi$ and $\phi$ would have to preserve $\J$. We denote by $\Aut(\D)_\J$ the subgroup of
$\Aut(\D)$ formed by automorphisms which preserve $\J$ and similarly by $\Der(\D)_\J
\subset \Der(\D)$ the subalgebra of derivations preserving $\J$.

\begin{lemma} 
There is an extension of Harish Chandra pairs $$1\rightarrow\langle GL_{\h}(r),\mathfrak{gl}_{\h}(r)\rangle\rightarrow\langle \Aut(\D,\M_r),\Der(\D,\M_r)\rangle\rightarrow\langle\Aut(\D)_{\J}, \Der(\D)_{\J}\rangle\rightarrow 1.$$
Moreover, $\Der(\D, \M_1) \simeq \frac{1}{h}\J$ where the right hand side is 
considered with commutator bracket. 
\end{lemma}

\medskip
\noindent 
\begin{proof} For $r=1$ this is proved in Corollary 3.1.8 of \cite{BGKP}. 
Since $\M_r \simeq \M_1^{\oplus r}$
and we have the map $Aut(\D, \M_1) \to Aut(\D, \M_r)$ sending $(\Phi, 
\Psi)$ to $(\Phi, \Psi^{\oplus r})$. Hence any lift for $\Phi \in Aud(D)_\J$
to $Aut(\D, \M_1)$ also gives a lift of the same $\Phi$ to $Aut(\D, \M_r)$. Therefore the right 
arrow is surjective. Its kernel is the group of automorphisms of $\M_r$ as a $\D$-module. Every 
such automorphism is uniquely determined by its value on generators, and we can choose 
a set of generators on which $y_j$ act by zero. Then their images are independent on $x_i$, 
which means that the automorphism if given by an invertible $r \times r$ matrix with entries in 
$\Ch$. 

The proof in the case of Lie algebras is entirely similar.
\end{proof}

\bigskip
\noindent
\textbf{Remark.}
Perhaps it will help the reader to understand the Lie algebra $\frac{1}{h} \J$ if we introduce
the grading in which $\deg h = 3$, $\deg y_j = 2$ and $\deg x = 1$ (then the Lie algebra will 
be an infinite direct product, not a direct sum, of its homogeneous components). 

Degree $-1$ component is spanned by the elements $\frac{1}{h}ad(y_1), \ldots, \frac{1}{h} ad(y_n)$, 
and this gives the non-integrable part of the Lie algebra. Degree $\geq 0$ part is the tangent Lie
algebra of $\Aut(\D, \M_1)$. Degree 0 component is spanned by 1 and $\frac{x_i y_j}{h}$, which gives
the tangent Lie algebra of the reductive part 
$\C^* \times GL(n)$. Degree $\geq 1$ part is the pro-nilpotent subalgebra
which contains, in particular, the elements of $h \C_h$.

Similar grading exists on $\Der(\D, \M_r)$ if we place $\mathfrak{gl}(r)$ in degree zero.

\bigskip
\noindent
The following diagram for Lie algebras and its group analogue  will be fundamental in our
analysis of the main Harish Chandra pair $\langle Aut(\D, \M_r), \Der(\D, \M_r)\rangle$:

$$
\xymatrix{0\ar"1,2"&\mathbb{C}_{\h}\ar"2,2"\ar"1,3"&\frac{1}{\h} \J\ar"2,3"\ar"1,4"&\Der(\D)_{\J}\ar"1,5"\ar@{=}"2,4"&0\\
0\ar"2,2"&\mathfrak{gl}_{\h}(r)\ar"2,3"&\Der(\D, \M_r)\ar"2,4"&\Der(\D)_{\J}\ar"2,5"&0}
$$
where the middle arrow uses sends $\alpha \in \frac{1}{h} \J$ to the pair $(\phi, \psi)$
with $\phi = [\alpha, \cdot]$ and $\psi = \alpha (\ldots)$. Note that in both cases the maps are
well defined, i.e. a possible denominator cancels out. Since the images of $\frac{1}{\h} \J$ 
and $\mathfrak{gl}_h(r)$ commute in $\Der(\D, \M_r)$, the following lemma is immediate
\begin{lemma}
$$
\Der(\D, \M_r) \simeq 
\big[ \frac{1}{\h} \J  \oplus 
\mathfrak{gl}_h(r)\big] / \mathbb{C}_{\h}
\simeq \frac{1}{\h} \J \oplus \mathfrak{pgl}_h(r). \qquad \square
$$
\end{lemma}

\bigskip
\noindent
In the group case we have a diagram 
$$\xymatrix{1\ar"1,2"&\mathbb{C}_{\h}^*\ar"2,2"\ar"1,3"&\Aut(\D, \M_1)\ar"2,3"\ar"1,4"&\Aut(\D)_{\J}\ar"1,5"\ar@{=}"2,4"&1\\
1\ar"2,2"&GL_{\h}(r)\ar"2,3"&\Aut(\D,\M_r)\ar"2,4"&\Aut(\D)_{\J}\ar"2,5"&1}$$
but the corresponding splitting fails in the constant terms with respect to $h$: 
$\Aut(\D, \M_r)$ has a subgroup $GL(r)$ which does not split as $PGL(r) \times \C^*$. 
However, there is a slightly different splitting on the group level, and the interplay 
between the two splittings is key to our later arguments on deformation quantization. 

\begin{lemma} We have a splitting 
$$
\Aut(\D, \M_r) \simeq \big[\Aut(\D, \M_1)/\C^*\big] \times 
\big[GL(r) \ltimes exp(h \cdot \mathfrak{pgl}_h(r))\big]
$$
where $exp(h \cdot \mathfrak{pgl}_h(r))$ is the pro-unipotent kernel of the 
evaluation map $PGL_h(r) \to PGL(r)$, $h \mapsto 0$. The two splittings (of the group and the 
Lie algebra) agree modulo $\langle \C^*, \C\rangle$: 
$$
\langle \Aut(\D, \M_r), \Der(\D, \M_r)/
\langle \C^*, \C \rangle \simeq
\langle \Aut(\D, \M_1)/\C^*, \frac{1}{h}\J/\C \rangle \times \langle PGL_h(r), \mathfrak{pgl}_h(r) \rangle 
$$
\end{lemma}

\medskip
\noindent
\begin{proof} The first statement is follows from the diagram before the lemma and the 
fact that the images of $GL_h(r)$ and $\Aut(\D, \M_1)$ in $\Aut(\D, \M_r)$
commute. To prove the splitting mod $\C^*$, observe the that pro-algebraic groups 
on both sides are semidirect product of finite dimensional reductive subgroups
and infinite dimensional pro-unipotent groups. The reductive part on the left hand side
is $GL(r) \times GL(n)$ where $GL(r)$ acts on the generators of $\M_r$ and 
$GL(n)$ on the variables $x_1, \ldots, x_n \in \D$, with the dual action on 
$y_1, \ldots, y_n$. The copy of $\C^*$ acts by scalar automorphisms on $\M_r$ only. Thus, 
the reductive part of the quotient on the left hand side is $PGL(r) \times GL(n)$. The 
same argument repeated for $r=1$ shows that the reductive part on the right is the same. 
The isomorphism of tangent Lie algebras on the pro-unipotent parts follows from the 
previous lemmas. Since we are taking quotients by the central copy of $\C^*$ which 
acts trivially on the Lie algebras, the semidirect products match as well.
\end{proof}

\medskip
\noindent
Informally we could say that on the group level there is an extra copy 
of $\C^* \subset GL(r)$ which on the level of Lie algebras migrates to 
the other factor 
$\C \subset \frac{1}{h} \J \simeq \Der(\D, \M_1)$ but 
after taking the quotient by these, the two splittings match.

\section {Transitive Harish Chandra Torsors}

The concept of a $G$-torsor over a smooth variety $X$ can be extended to Harish-Chandra
pairs. We will only need the special case of a \textit{transitive} Harish Chandra torsor. 
Note that the definition below implies that $\dim(\mathfrak{h}/\mathfrak{g})$ is
finite and equal to $\dim_\C Y$, as is the case with the main Harish Chandra
pair considered in the previous section. 

\begin{defn}
A \textit{transitive Harish-Chandra torsor} (or tHC torsor for short) over a Harish Chandra pair $(G, \mathfrak{h})$ on  a smooth variety $Y$ 
is a $G$-torsor $P \to Y$ together with a Lie algebra morphism
$\mathfrak{h} \to \Gamma(P, T_P)$ which induces an isomorphism of vector 
bundles $h \otimes_\C \mathcal{O}_P \simeq T_P$. 

This data can be rephrased in terms of the Atiyah algebra $\LL(P)$
of $G$-invariant vector fields on $P$, viewed as a Lie algebroid on $Y$
$$
0 \to Ad(P) = P_{\mathfrak{g}} \to \LL(P) \to T_X \to 0
$$
For a tHC torsor $P$ we have  an isomorphism of locally free
sheaves on $Y$:
$$
A: \LL(P) \simeq P_{\mathfrak{h}}
$$
which restricts to identity on $Ad(P) = P_{\mathfrak{g}}$. This isomorphism 
does not agree with the Lie bracket but instead $P_{\mathfrak{h}}$ has a 
flat algebraic connection such that
$$
A[x, y] - [Ax, Ay] = \partial_x Ay - \partial_y Ax 
$$
where $x \mapsto \partial_x$ is the anchor map of $\LL(P)$.
\end{defn}

\medskip
\noindent 
\textbf{Remarks.} 

(1) Note that the sheaf of sections of $P_{\mathfrak{h}}$ 
has a Lie bracket since the action of $G$ preserves the bracket on $\mathfrak{h}$.

(2) In general, if $V$ is a representation of $G$ with the infinitesimal 
action of $\mathfrak{g}$
extended to that of $\mathfrak{h}$, for a transitive Harish-Chandra torsor $P$
the associated vector bundle $P_V = P \times_G V$ has a flat algebraic
connection, see Section 2.3 in in \cite{BK}.

Let $X$ be a $2n$-dimentional symplectic variety over $k$ with a symplectic form $\Omega$. 
Assume that a deformation quantization $\mathcal{O}_h$ of $\mathcal{O}_X$ is 
given, such that its non-commutative period is given by $\omega(t) = [\Omega] 
+ t \omega_1 + t^2 \omega_2 + \ldots \in H^2_{DR}(X) \otimes \Ch$. 

Assume that a smooth Lagrangian subvariety $Y \subset X$ is given, with a rank $r$ vector bundle
$E$ which we consider as a sheaf of $\mathcal{O}_X$-modules via the direct image functor. 
Assume further that $x \in Y$ is a point and in some neighborhood of $x \in X$ we are 
given a deformation quantization $\Eh$ of $E$. 

The stalk $\mathcal{O}_{h, x}$ is a non-commutative ring with a maximal ideal $\mathfrak{m}_x$, 
which is the preimage of the maximal ideal $\mathfrak{n}_x$ 
in the commutative local ring 
$\mathcal{O}_{h, x}/h \mathcal{O}_{h, x} \simeq \mathcal{O}_{X, x}$. 
Let $\widehat{\mathcal{O}}_{h, x}$ and $\widehat{\mathcal{E}}_{h, x}$ be completions of 
stalks with respect to this maximal ideal. 

\begin{lemma}
There exist isomorphisms $\eta: \widehat{\mathcal{O}}_{h, x} \simeq \D$, $\mu: \widehat{\mathcal{E}}_{h, x} \simeq \M_r$ of topological $\Ch$-modules 
which are compatible with filtrations and action of rings on corresponding modules.
Moreover, the first isomorphism may be chosen in such a way that the images of 
$y_1, \ldots, y_n$ in $\widehat{\mathcal{O}}_{X, x}$ come from 
a regular sequence in 
$\mathcal{O}_{X, x}$ generating the ideal $I_x$ of functions vanishing on $Y$. 
\end{lemma}

\begin{proof}
We sketch a proof here briefly. First, we can find an isomorphism 
$$
\widehat{\mathcal{O}}_{X, x} = \C[[x_1, \ldots, x_n, y_1, \ldots, y_n]]
$$
such that $y_1, \ldots, y_n$ are the images of a regular sequence
defining $Y$ in the neighborhood of $x$. This is due to the 
formal Weinstein Lagrangian Neighborhood Theorem, the proof
of which, cf. Sections 7,8 
on \cite{CdS}, may be used without changes. The key point here is that 
the Moser Trick in Section 7.2 of \textit{loc. cit.} works in the 
completion of the local ring, although not the local ring itself.

Next, both $\D$ and $\widehat{\mathcal{O}}_{h,x}$ can be viewed
as deformation quantizations of the algebra $\C[[x_1, \ldots, x_n, 
y_1, \ldots, y_n]]$ with the Poisson bivector equal to 
$h (\sum_i \partial/\partial x_i \wedge \partial/\partial y_i) + h^2 \alpha(h)$.
By general deformation theory, the bivector $\alpha(h)$ is a
Maurer-Cartan solution for the DG Lie algebra of polyvector fields with 
the non-trivial differential given by the bracket with the bivector
$\sum_i \partial/\partial x_i \wedge \partial/\partial y_i$. If we identify
polyvector fields with differential forms, this bracket will become the
standard de Rham differential. Since for the algebra of formal power series the 
the Rham complex is exact and the Maurer-Cartan groupoid is invariant 
under quasi-isomophisms, all deformation quantizations are equivalent and 
$\D$ is isomorphic to the completion of $\mathcal{O}_{h, x}$. This
settles the assertions about the algebras. 

As for the module part, since $(\Eh/h \Eh)_x$ is free over $\mathcal{O}_{Y, x}$ we can find an isomorphism of this module with $M_r = \M_r/h \M_r \simeq \C[[x_1, \ldots, x_n]]^{\oplus r}$. Now we repeat the argument of Lemma 2.3.5 
(given there for $r=1$) to produce an isomorphism of $\widehat{\mathcal{E}}_{h, x}$ and $\M_r$.
\end{proof}

\bigskip
\noindent
Given the data $(X, Y, \mathcal{O}_h)$ we can consider the proalgebraic scheme $P_\J$ 
parameterizing the pairs $(x, \eta)$, where $x \in Y$ and $\eta$ is an isomorphism as above, cf. Section 3.1 of \cite{BK} for a similar case. 

\begin{lemma}
$\mathcal{P}_{\J}$ has a canonical structure of a transitive 
$\langle \Aut(\D)_{\J},\Der(\D)_{\J}\rangle$-torsor.
\end{lemma}

\begin{proof}
The group $\Aut(\D)_\J$ acts on $P_\J$ 
by changing the isomorphism $\eta$.  Locally in 
the Zariski topology on $X$ we can denote by $A$, $B$ the rings of sections of 
$\mathcal{O}_X$, $\Oh$ over an affine subset, and then the pair $(x, \eta)$
can be described by two ring homomorphisms $x:  A \to \C$, $\eta: 
\lim_k (B/\mathfrak{n}^k)  \simeq \D$ where $\mathfrak{n} \subset B$ 
is the preimage in $B$
of the ideal $Ker(x) \subset A = B/hB$.  Given a derivation $\partial: \D \to \D$, 
we can consider its composition with $\D \to \C = \D/\mathfrak{m}$ which 
descends to a linear function on $\mathfrak{m}/\mathfrak{m}^2$, and using
the isomorphism $\eta$ and $h$-linearity of $\partial$, to a linear function 
$z: Ker (x)/(Ker (x))^2 \to \C$. Thus we can write extensions of $x$ and $\eta$
over the ring $\C[\epsilon]/(\epsilon^2)$ of dual numbers 
$$
x_\epsilon (a) :=  x(a) + \epsilon z(a); \qquad \eta_\epsilon(b) 
= \eta(b) + \epsilon \partial(\eta(b)). 
$$
This means that $\Der(D)_\J$ maps to vector fields on the torsor $P_\J$. The 
defining properties of tHC torsors (agreement with the group action and the 
bracket on vector fields) are immediate from the definitions.
\end{proof}

\begin{prop}
A choice of a vector bundle $E$ of rank $r$ on $Y$ and its deformation quantization 
$\mathcal{E}_{\h}$ is equivalent to a choice of a lift of the torsor $\mathcal{P}_{\J}$ 
to a transitive Harish-Chandra $\langle \Aut(\D,\M_r, \Der(\D,\M_r))\rangle$-torsor 
$\mathcal{P}_{\M_r}$ along the extension of pairs
$$1\rightarrow\langle GL_{\h}(r),\mathfrak{gl}_{\h}(r)\rangle\rightarrow\langle \Aut(\D,\M_r),\Der(\D,\M_r)\rangle\rightarrow\langle\Aut(\D)_{\J}, 
\Der(\D)_{\J}\rangle\rightarrow 1.$$
\end{prop}

\begin{proof}
Given $\Eh$, we can construct the tHC torsor $P$ for which the 
fiber at $x \in Y$ represents all pairs 
$(\eta, \mu)$ which identify completions of $\Oh$ and $\Eh$ at $x$ with the
``formal local models'' $\D$ and $\M_r$. It has the transitive Harish Chandra
structure similarly to $P_\J$.
By consturction, the torsor $P_\J$ is 
induced from $P$ by forgetting the isomorphism $\mu$. Conversely, given a lift
$P$ of the tHC torsor $P_\J$, we have a vector bundle $P_{\M_r}$ associated 
via the action of $\langle \Aut(\D, \M_r), \Der (\D, \M_r) \rangle$ on 
$\M_r$. It carries a flat algebraic connection as any vector bundle associated to 
a tHC torsor. As in Lemma 3.4 of \cite{BK}, 
the Zariski sheaf $\Eh$ is can be recovered
as the sheaf of flat sections of $P_{\M_r}$. 
\end{proof}

\section{Lifting the torsor by steps.}

In this section we study the problem of tHC orsor lifting by considering a chain of surjections
$$
\langle G,  \mathfrak{h} \rangle
\to \langle G_2,  \mathfrak{h}_2 \rangle
\to \langle G_1,  \mathfrak{h}_1 \rangle
\to \langle G_0,  \mathfrak{h}_0 \rangle
$$
At the two ends we have pairs 
$$
\langle G,  \mathfrak{h} \rangle = \langle \Aut(\D,\M_r),\Der(\D,\M_r)\rangle, \qquad
\langle G_0,  \mathfrak{h}_0 \rangle = \langle \Aut(\D)_\J,\Der(\D)_\J\rangle
$$
and the intermediate pairs both have product decompositions
$$
\langle G_2, \mathfrak{h}_2 \rangle 
= \langle G, \mathfrak{h} \rangle/
\langle \C^*, \C \rangle \simeq
\langle \Aut(\D, \M_1)/\C^*, \frac{1}{h}\J/\C \rangle \times \langle PGL_h(r), \mathfrak{pgl}_h(r) \rangle. 
$$
$$
\langle G_1, \mathfrak{h}_1 \rangle =
\langle \Aut(\D)_\J, \Der(\D)_\J \rangle \times \langle PGL_h(r), \mathfrak{pgl}_h(r) \rangle. 
$$
The following is a more precise version of Theorem 1.
\begin{theorem}
Let $P_0= P_\J$ be the transitive Harish-Chandra torsor over the pair $\langle G_0, \mathfrak{h}_0\rangle$
on $Y$, induced by the choice of quantization $\Oh$ on $X$ and a closed embedding $j: Y \to X$ of a 
smooth Lagrangian subvariety $Y$. Assume that $\Oh$ has the non-commutative period
$$
[\omega] + h \omega_1 + h^2 \omega_2 + \ldots \in H^2_{DR}(X)[[h]]
$$
Then $j^* \omega_1$ admits a canonical lift to $H^2_F(Y)$ represented by the class
of the Picard algebroid $\LL(\Oh)$. 
\begin{enumerate}
\item The groupoid category of lifts of $P_0$ to a tHC torsor $P_1$ over 
$\langle G_1,  \mathfrak{h}_1 \rangle$ is equivalent to the category of 
$PGL_h(r)$-bundles on $Y$ with a flat algebraic connection.
\item Given a choice of $P_1$, the category of its lifts to a tHC torsor $P_2$ over
$\langle G_2,  \mathfrak{h}_2 \rangle$ is equivalent to the category of lifts of 
the original torsor $P_0$ to a tHC torsor over  
 $\langle \Aut(\D, \M_1)/\C^*, \frac{1}{h}\J/\C \rangle$. The latter is non-empty if and only if 
 in $H^2_{DR}(Y)[[h]]$ one has
$$
j^*(h^2 \omega_2 +  h^3 \omega_3 + \ldots) = 0;
$$ 
\item Given a choice of $P_2$, the groupoid category of its lifts to a 
tHC
torsor $P$ over $\langle G,  \mathfrak{h} \rangle$ is equivalent to the category of 
rank $r$ modules $E$ over the Lie algebroid $\mathcal{L}^+(\Oh)$, equipped with an isomorphism 
of $\mathbb{P}(E)$ and the flat $PGL(r)$-bundle 
induced from $P_1$ via the homomorphism $PGL_h(r) \to PGL(r)$. Such lifts exist
if and only if the following equality holds in $H^2_F(Y)$:
$$
\frac{1}{r} c_1(E) = \frac{1}{2} c_1(K_Y) + c(\LL(\Oh)) = c(\LL^+(\Oh)).
$$

\item For a given choice of $P_1$, the groupoid category of its lifts to a tHC torsor
$P$ over $\langle G_1,  \mathfrak{h}_1 \rangle$ - if non-empty - is 
equivalent to the groupoid category $\mathcal{O}_Y[[h]]^*$-torsors with a flat 
algebraic connection. More precisely, for any fixed choice of $P$ and a flat 
$\mathcal{O}_Y[[h]]^*$-torsor $L$ there is a well-defined tHC torsor $L \star P$, 
and the functor $L \mapsto L \star P$ gives an equivalence of categories. 
\end{enumerate}
\end{theorem}

\bigskip
\noindent
\textit{Proof.}
The first part is easy due to the splitting of the Harish Chandra pair 
$\langle G_1, \mathfrak{g}_1 \rangle$. We are looking for a torsor $P_1$ over $G_1$ 
and a $G_1$-equivariant form $\mathfrak{h}_1$-valued form on the total space of $P_1$, 
which lifts a similar form on the total space of $P_0$. Due to the splitting of $G_1$
such $P_1$ would be a fiber product over $Y$, of the original $P_0$ and a $PGL_h(r)$-torsor 
$Q$. The connection form on $P_1$, due to the $G_1$-equivariant splitting of $\mathfrak{h}_1$,
would also have to be a sum of the $\mathfrak{h}_0$-valued connection on $P_0$ and 
a $\mathfrak{pgl}_h(r)$-valued connection on $Q$. The zero curvature condition on $P_1$
(i.e. agreement with the Lie bracket) implies that the connection on $Q$ is flat.

\bigskip
\noindent
For the second part, we use the splitting of the Harish Chandra pair 
$\langle G_2, \mathfrak{h}_2 \rangle$ and the fact that the flat bundle $Q$ 
is already defined on the previous step. Therefore, $P_2$ would be a fiber product 
of $Q$ and a tHC torsor $R$ over $\langle \Aut(\D, \M_1)/\C^*, \frac{1}{h}\J/\C \rangle$
lifting $P_0$. Comparing the definitions we see that a tHC structure on $P_2$ of the 
required type is equivalent to a similar structure on $R$ lifting that of $P_0$. 
Existence  of $R$ follows by a combination of Proposition 2.7 in \cite{BK} and 
 Lemma 5.2.2 of \cite{BGKP}.

\bigskip
\noindent
For the third part we use the short exact sequence
$$
1 \to \langle \mathbb{C}^*, \mathbb{C} \rangle \to \langle G,  \mathfrak{h} \rangle \to
\langle G_2,  \mathfrak{h}_2 \rangle \to 1.
$$
First consider the group side. Due to the splitting of Lemma 6 on the level of 
usual torsors, we just need to lift a $PGL_h(r)$ torsor $Q$ to a torsor  $\widetilde{Q}$
over $GL(r) \ltimes exp(h \cdot \mathfrak{pgl}_h(r)$. The required Harish Chandra structure is
an isomorphism of bundles $T_{P} \simeq \mathfrak{h}\otimes_\C \OO_{P}$
which is $G$-equivariant and has the zero curvature condition.
Since the action of $\C^* \subset GL(r)$
on $h\cdot \mathfrak{pgl}_h(r)$ and $\mathfrak{h}$ is trivial, we can take $\C^*$ invariant direct 
image of both sheaves under $P \to P_2= P/\C^*$ where they have a $G_2$-equivariant structure, and 
look for a $G_2$-equivariant isomorphism on the total space of $P_2$ instead. 

The tangent bundle of $P$ turns into the Atiyah algebroid of the $\C^*$-torsor $P \to P_2$
and the trivial bundle with fiber $\mathfrak{h}$ also gets a Lie algebroid structure since
on $P_2$ we have identified $T_{P_2}$ with the trivial bundle with fiber $\mathfrak{h}_2
\simeq \mathfrak{h}/\C$. We need to construct a $G_2$-equivariant isomorphism of two bundles
with an additional property that corresponds to agreement of brackets (zero curvature) on $P$.
The $G_2$-equvariant stucture means that the projection to 
$T_{P_2} \simeq \mathfrak{h}_2 \otimes_\C \OO_{P_1}$
has a partial section over the smaller sub-bundle $\mathfrak{g}_2 \otimes_\C  \OO_{P_1}$.
Taking the quotient of both algebroids by the image of this sub-bundle, and then taking
$G_2$-equivariant direct image to $Y$, we reduce to the question of isomorphism of two 
Picard algebroids on $Y$. Both are classified by a cohomology class in $H^2_F(Y)$, hence 
the isomorphism exists precisely when the two classes are equal. 

To compute the class for the algebroid obtained from $\mathfrak{h} \otimes \OO_P$ 
observe that instead of taking the equivariant descent with respect to $G_2$, we can 
first take the descent by $PGL_h(r)$ and then by $\Aut(\D, \M_1)/\C^*$. 
The first step will lead to an equivariant Picard 
algebroid on $R$ with the fiber corrseponding to the middle term of the Lie algebra
extension
$$
0 \to \C \to \Der(\D, \M_1) \to \Der(\D, \M_1)/\C \to 0 
$$ 
In other words, we find ourselved in the rank 1 situation. 
By Proposition 4.3.5 in \cite{BGKP} its equivariant 
descent to $Y$ gives a Picard algebroid with the class 
$\frac{1}{2} c_1(K_Y) + j^* \omega_1$. 

For the Atiyah algebroid of $P \to P_2$ we will first take the descent with respect to 
$\Aut(\D, \M_1)/\C^*$, then with respect to $exp(h \cdot \mathfrak{pgl}_h(r)$.
 This leads to a $GL(r)$-torsor $E$ on $Y$ lifting 
a $PGL(r)$-torsor $\mathbb{P}(E)$. The equivariant descent of the Atiyah algebroid
of $P \to P_2$ to a Picard algebroid on $Y$ is just the quotient of the Atiyah algebroid 
of $E$ by the trace zero part in $End(E) \subset \LL(E)$. 
Therefore it has class $\frac{1}{r} c_1(E)$ and the desired equation in $H^2_F(Y)$ is 
$\frac{1}{r}c_1(E) = \frac{1}{2} c_1(K_Y) + j^* \omega_1$, concluding the proof of 
the third part.

\bigskip
\noindent
For the last part, we interpret the tHC torsor lifting in the language of gerbes, 
following Chapter 5 of \cite{B}. Consider the central extension 
$$
1 \to \langle \C_h^*, \C_h \rangle \to \langle G, \mathfrak{h} \rangle \to 
\langle G_1, \mathfrak{h}_1 \rangle \to 1
$$
The lifting of a tHC torsor $P_1 = Q \times_Y P_0$ to $P$, can be split as follows. First,
choose a lifting on the group level. As we have seen before, globally on $Y$ this may not  
be possible but local lifts form a gerbe over $\OO_Y[[h]]^*$, see 
Definition 5.2.4 of \cite{B}. If a (local) lift is chosen, we can look for an $\mathfrak{h}$-valued
connection on the lift, and all possible choices of such a connection 
form a connective structure on the gerbe in the sense of Definition 5.3.1
of \textit{loc. cit.}. Moreover, whenever we can choose 
a connection, this leads to the curvature $d \Omega + \frac{1}{2} \Omega \wedge \Omega$ where $\Omega$ is the $\mathfrak{h}$-valued form on the total space
of the torsor which described the connection.  This gives a curving of the
connective structure, see Definition 5.3.7 of 
\textit{loc. cit}. As in Theorem 5.3.17 of \textit{loc. cit.} this leads to a gerbe over the $\C_h$-version of Deligne complex
$$
\mathfrak{Del}_h(Y) : = \OO_Y[[h]]^* \to \Omega^1_Y[[h]] \to \Omega^2_Y[[h]] \to \ldots
$$
where the first arrow sends $f$ to $\frac{df}{f}$. 
Since we want to trivialize the gerbe over this complex (i.e. the lifting torsor must exist
\textit{and} it must admit a connection \textit{and} the connection must have zero 
curvature), the category of such trivializations - if nonempty - is equivalent to the category 
of torsors over the same complex. By a $\Ch$ version of Theorem 2.2.11
in \textit{loc. cit.} the latter is equivalent to the category of  $\OO_Y[[h]]^*$-torsors with a flat algebraic
connection. In particular the set of equivalence classes of
deformation quantizations with 
a fixed projectivization  $Q$ is in bijective correspondence with 
the set of isomorhism clases of $\OO_Y[[h]]^*$-torsors with a flat algebraic
connection.

A more explicit, if a bit less conceptual version of this step is as follows. We
look at the short exact sequence
$$
\Omega^{\geq 1}(Y)[[h]] \to \mathfrak{Del}_h(Y) \to \OO_Y[[h]]^*
$$
and piece of the associated long exact sequence in cohomology:
$$
\ldots \to H^1(Y, \OO_Y[[h]]^*) \to H^2_F(Y)[[h]] \to H^2(Y, \mathfrak{Del}_h(Y))
\to H^2(Y, \OO_Y[[h]]^*) \to \ldots...
$$
If a $\langle G_1, \mathfrak{h}_1 \rangle$ torsor and its tHC structure can be 
trivialized on an covering $\{U_i\}$ of $X$ then we can first look at
$G$-valued liftings $\psi_{ij}: U_i \cap U_j \to G$ of the cocycle defining $P_2$.
Then we can look at those $\mathfrak{h}$-valued connections $\nabla_i$
on the trivial  $G$-bundles on each $U_i$ which lift the given 
$\mathfrak{h}_1$-valued connection on $P_1$. 

The functions $\psi_{ij}$ lead to an expression $a_{ijk} = \psi_{ij} \psi_{jk} \psi_{ki}$
on $U_i \cap U_j \cap U_k$ which takes values in $\OO_Y[[h]]^*$ since
its projection to $G_2$ is trivial. Moreover, the difference of $\nabla_i - \nabla_j$
on $U_i \cap U_j$ is given by a $\Ch$-valued differential 1-form $b_{ij}$
and a quick comparison of definitions shows that on triple intersections 
$(b_{ij} + b_{jk} + b_{ki}) = d a_{ijk}/a_{ijk}$. On each $U_i$ the 
curvature of $\nabla_i$ is given by a 2-form $c_i$ which takes values in 
the Lie subalgebra $\Ch \subset \mathfrak{h}$ since the projection of the
connection to $\mathfrak{h}_2$ satisfies the zero curvature conditon. On 
the double intersections we have $(c_i - c_j) = d b_{ij}$, the usual comparison 
of curvatures for two different connections. 

The two identities relating $b_{ij}$ with $a_{ijk}$ and $c_i$ with $b_{ij}$ mean that
the three groups of sections define a single cohomology class in 
$H^2(Y, \mathfrak{Del}_h(Y))$. 

The projection of this class to $H^2(Y, \OO_Y[[h]]^*)$ is represented by 
the cocycle $a_{ijk}$. In the case when the class is trivial we can readjust 
the choice of $\psi_{ij}$  by the cochain resolving $a_{ijk}$
and ensure $a_{ijk} \equiv 1$. This means that
$\psi_{ij}$ do define a lift of $P_1$ to a $G$-torsor $P$. Now the obstruction too is 
lifted from $H^2(Y, \mathfrak{Del}_h(Y))$ to $H^2_F(Y)[[h]]$. The fact that
such lift is only well defined up to an element in the image of  
$H^1(Y, \OO_Y[[h]]^*)$ reflects the fact that $P$ is only well defined only 
up to a twist by a torsor over $\OO_Y[[h]]^*$.  Since $P \to P_1$ is 
a $G_1$-equivariant torsor over $\Ch^*$ on $P_1$ we can descend its Lie algebroid to $Y$
and obtain a class of this algebroid in $H^2_F(Y)[[h]]$ which must vanish 
if we want to lift the tHC structure to $P$.  This finishes the proof. $\square$

\section{Remarks and open questions.}

\begin{itemize}

\item When $Y$ is affine, a straightforward argument following Section 3
and Appendix A from \cite{Y} or, alternatively, an imitation of the 
arguments in Section 5 of \cite{CCT}, shows that we can assume that $\Eh$ 
is isomorphic to $E[[h]]$ as a Zariski sheaf of $\C_h$-modules. This 
is established inductively, considering $\Eh/h^k \Eh$ and observing 
that obstructions to trivializing $\Eh/h^{k+1} \Eh$, given a trivialization 
of $\Eh/h^k \Eh$, live in an $Ext$ group which vanishes since $Y$ is
affine and $E$ is locally free on $Y$. 

\item Our results make sense in the category of complex manifolds 
where de Rham cohomology groups correspond to the holomorphic de Rham 
complex. 
In fact, the arguments of the paper carry over to the case of etale 
topology, and should hold in the  case of smooth Deligne-Mumford stacks 
as well. 

For the case of real manifolds more work is needed to adapt our arguments
but one expects that the output will be similar in spirit 
 to the arguments in \cite{NT}.

\item Perhaps the most celebrated example of a vector bundle with a flat projective 
connection is  the bundle of conformal blocks on the moduli space $\mathfrak{M}$ of (marked) curves with the Hitchin connection.
It would be interesting to see whether $\mathfrak{M}$ may be realized as a Lagrangian
subvariety (actually, sub-orbifold) in a larger sympletic variety, so that the bundle
of conformal blocks admits a deformation quantization. Maybe the approach to 
$\mathfrak{M}$ via representations of $\pi_1(C)$ in $PSL(2, \mathbb{R})$
(which are embedded into $PSL(2, \mathbb{C})$ representations) could give
something here.

\item In theory, our methods should work for vector bundles on smooth coisotropic
subvarieties. In this case the conormal bundle $N^*$ embeds as a null-foliation
sub-bundle of the tangent bundle: $N^* \simeq T_F \hookrightarrow T_Y$. 
Consequently, the full de Rham complex will be replaced by the normal de Rham complex
built from exterior powers of $T^*_F \simeq N$. For second order deformations 
this
has been studied in \cite{P}.

\item It has been noted in \cite{P} that for a given bundle $E$ on $Y$ its
 deformation quantization problem is described by a sheaf of \textit{curved} 
 dg Lie algebras. If we assume that all $\omega_i$ restrict to zero on $Y$, 
 we have an interesting situation: 
at the first step of the deformation there is an obstruction (coming from the curvature
element) but if it is resolved, there are no further obstructions. One way to understand it
is to view $\Eh$ as a deformation of an $\Oh$ module $E$ over the power series
ring $\C[[h]]$. Not every deformation is a deformation quantization though, 
and one way to formulate the condition is to require that the first order
deformation is given by an element of $Ext^1_{\Oh}(E, E)$ which 
projects to identity under the map 
$$
Ext^1_{\Oh}(E, E) \to Hom_{\mathcal{O}_X}(Tor_1^{\Oh}(\mathcal{O}_X, E), E)
= Hom_{\mathcal{O}_X}(E, E)
$$
provided by the Change of Rings Spectral Sequence for the homomorphism
$\Oh \to \mathcal{O}_X$. This implies that the identity $Id_E$ must be closed 
with respect to this spectral sequence differential
$Hom_{\mathcal{O}_X}(E, E) \to Ext^2_{\mathcal{O}_X}(E, E)$. 
Once this obstruction  vanishes, any higher order extension of the first 
order deformation as an $\Oh$-module, is automatically a deformation 
quantization of $E$ and the first order adjustment to the differential of 
the deformation complex
in \cite{P} removes the curvature.

\item Perhaps the category of vector bundles with a structure of a $\LL^+(\Oh)$-modules
deserves a closer attention. The condition of having a flat structure on $\mathbb{P}(E)$ and
an equation in $H^2_F(Y)$
$$
\frac{1}{r} c_1(E) = \frac{1}{2} c_1(K_Y) + j^* \omega_1
$$
is stable under direct sums, and taking the tensor product of $E$ with a flat vector 
bundle $F$. For any pair of bundles $E_1, E_2$ with these conditions the bundle 
$Hom_{\mathcal{O}_Y}(E_1, E_2)$ is flat. Indeed, for a local section 
$\varphi$ of this bundles we can attempt to take its derivative along 
a vector field $\partial$ by lifting it to a section of 
$\LL^+(\Oh)$ and then using the action of $\LL^+(\Oh)$ of $E_1$ and
$E_2$. The lift of $\partial$ is only well-defined up to a section of 
$\mathcal{O}_Y \subset \LL^+(\Oh)$ but $\varphi$ is $\mathcal{O}$-linear
so its derivative will not depend on the choice of this lift.

In particular, if
we can find a line bundle $L$ which satisfies the above equation with $r=1$, then 
we can write any $\LL^+(\Oh)$-module in the form $E = F \otimes L$ where $F$ has flat 
algebraic connection. Of course, when $\frac{1}{2} c_1(K_Y) + j^* \omega_1 = 0$ we can 
take $L = \OO_Y$. 
Another instance is when $j^* \omega_1 = 0$, in which case the 
category admits an involution $E \mapsto E^* \otimes K_Y$. When, 
in addition, we can find $L$ such that $K_Y \simeq L^{\otimes 2}$
this involution corresponds to dualization of the local system $F$. 

In would be interesting to compare the de Rham cohomology of the 
local system $Hom_{\mathcal{O}_Y}(E_1, E_2)$ and 
the groups $Ext_{\Oh}(\mathcal{E}_{1, h}, \mathcal{E}_{2, h})$ for 
deformation quantizations of $E_1$, $E_2$, respectively.

\item The original motivation of \cite{BG} was to relate 
deformation quantization of the Kapustin-Rozansky 2-category
of the original algebraic symplectic variety $(X, \omega)$, cf. \cite{KR}. 
At least when the class $\omega_1 \in H^2_{DR}(X)$ vanishes, it is 
natural to expect that an $\LL^+(\Oh)$-module $E$ on a Lagrangian 
$Y$ should define an object $(Y, E)$ in this category. We further expect
that for the same $Y$ and different $E_1$, $E_2$ the cyclic homology 
of the 1-category $Hom((Y, E_1), (Y, E_2))$ has something to do 
with the de Rham cohomology of the local system $Hom_{\mathcal{O}_Y}
(E_1, E_2)$. 

This also suggests a connection between $\mathcal{L}^+(\Oh)$-modules
and generalized complex branes of Gualtieri, \cite{G}. Indeed, those are defined 
as modules over a Lie algebroid which appears after restriction to a
subvariety. On the other hand, counting  dimensions we see that
$\mathcal{L}(\Oh)$ cannot come from an exact Courant algebroid on $X$, 
but perhaps one should work with some algebroid on $X \times Spec(\C[[h]])$. 
By an earlier remark in this section one expects a more general construction
for vector bundles on smooth coisotropic subvarieties when connections
and similar structures are only defined along the null-folitation. 
\end{itemize}

\noindent
\textit{Address: Department of Mathematics, UC Irvine, 340  Rowland Hall, Irvine CA 92617, USA}

\end{document}